\documentclass[draft]{ifacconf}

\usepackage{graphicx}      % include this line if your document contains figures
\usepackage{natbib}        % required for bibliography
\let\theoremstyle\relax
\usepackage{amsthm}
\usepackage{amsfonts}
\usepackage{amsmath}
\usepackage{graphicx} 
\usepackage{epsfig}
\usepackage{bbm}
\usepackage{epstopdf}
\usepackage[strings]{underscore}
\usepackage{xcolor}
\newtheorem{theorem}{Theorem}[section]

\newtheorem{lemma}[theorem]{Lemma}
\newtheorem{proposition}{Proposition}

\theoremstyle{definition}

\DeclareMathOperator{\T}{\mathcal{T}}

\DeclareMathOperator{\R}{\mathbb{R}}

\newcommand{\I}{{\mathcal I}}

\newcommand{\lra}{\longrightarrow}
\newcommand{\ra}{\rightarrow}

\newcommand{\cI}{{\mathcal I}}

\newcommand{\cA}{{\mathcal A}}

\newcommand{\cC}{{\mathcal C}}

\newcommand{\cU}{{\mathcal U}}

\newcommand{\cT}{{\mathcal T}}

\renewcommand{\phi}{\varphi}
\renewcommand{\a}{{\alpha}}

\renewcommand{\l}{\lambda}

\newcommand{\BUC}{\operatorname{BUC}}

\begin{document}
\begin{frontmatter}

\title{A hybrid control framework for an optimal visiting problem} 

\thanks[footnoteinfo]{A. Festa was supported by MIUR grant ``Dipartimenti Eccellenza 2018-2022" CUP: E11G18000350001, DISMA, Politecnico di Torino}

\author[Bagagiolo]{Fabio Bagagiolo} 
\author[Festa]{Adriano Festa}
\author[Marzufero]{Luciano Marzufero}

\address[Bagagiolo]{Dipartimento di Matematica,
Universit\`a di Trento, 
Via Sommarive, 14, 38123 Povo (TN) Italy, (e-mail: fabio.bagagiolo@unitn.it).}
\address[Festa]{Dipartimento di Scienze Matematiche ``G. L. Lagrange", 
Politecnico di Torino, 
Corso Duca degli Abruzzi, 24, 10129 Torino Italy, (e-mail: adriano.festa@polito.it)}
\address[Marzufero]{Dipartimento di Matematica,
Universit\`a di Trento, 
Via Sommarive, 14, 38123 Povo (TN) Italy, (e-mail: luciano.marzufero@unitn.it).}

\begin{abstract}                % Abstract of not more than 250 words.
The optimal visiting problem is the optimization of a trajectory that has to touch or pass as close as possible to a collection of target points. The problem does not verify the dynamic programming principle, and it needs a specific formulation to keep track of the visited target points. In this paper, we introduce a hybrid approach by adding a discontinuous part of the trajectory switching between a group of discrete states related to the targets. Then, we show the well-posedness of the related Hamilton-Jacobi problem, by reformulating the optimal visiting as a collection of time-dependent optimal stopping problems. 
\end{abstract}


\end{frontmatter}
%===============================================================================

\section{Introduction}

In this paper, we study the problem of optimizing a trajectory to pass as close as possible to a collection of target sets at a certain time. We call this problem \emph{optimal visiting}, and it is related to the  ``Traveling Salesman Problem'', including the high complexity of computation for a large number of targets.  Furthermore, the dynamical nature of the problem poses some additional difficulties, in particular for the study of the related Hamilton-Jacobi (HJ) equations. As observed in \cite{bagben}, to recover the dynamic programming property and hence HJ, it requires a special framework able to include a ``memory" of the targets already visited. This can be done using various tools. In  that paper, a sort of continuous memory was introduced, and the problem was studied in the framework of dynamic programming and HJ equations. A switching/discontinuous/hybrid memory was instead used for a one-dimensional optimal visiting problem on a network in \cite{bagfagmagpes}. 

In the present paper, for a multi-dimensional problem, we propose a hybrid control-based construction, similarly as in \cite{bagfagmagpes}, with the difference that one can get rid of a target at any moment just paying a suitable cost. This will lead to an optimal-stopping formulation of the problem. Here we focus on the theoretical results that are sufficient to guarantee the well-position of the problem and the characterization of the value function as the unique solution of a suitable HJ problem. However, an application of the current framework is discussed in \cite{BFMProc2}, where some numerical results are also reported. Moreover a generalization of the idea to a mean-field games related model - i.e., where an infinity of self-similar agents optimize their trajectories - is developed in \cite{BFMNoDea}. For that possible generalization, we consider time-dependent optimal stopping problems, that is, with running cost and stopping cost explicitly dependent on time. See also the comments on \S\ref{conclusion}.

We use the theory of viscosity solutions (see, e.g., \cite{BardiCapuz@incollection, Festa2017127}). Moreover, the hybrid framework is strictly related to hybrid control (see \cite{branicky1998unified} and also \cite{bensoussan1997hybrid, dharmatti2005hybrid}).
%For switching hybrid control problems and differential games, related to the model here presented, we refer to \cite{bagdan} and to \cite{bagmagzop}. 

\section{The optimal visiting problem}

Given $N$  disjoint compact target sets $\{\T_j\}_{j=1,\ldots,N}\subset \R^d$, we represent the state of the system by the pair $(x, p)\in\R^d\times\I$, where $p=(p^1, p^2,\ldots, p^N)\in\cI=\{0,1\}^N$. Therefore, $x$ is the continuous state variable (i.e., the position in $\R^d$) and $p$ is the switching discrete state variable. The controlled dynamical system is
\begin{equation}
\label{eq_stato}
\begin{cases}
 y'(s)=f(y(s),\alpha(s), q(s)),&\text{a.e.}\ s\in ]t,T]\\
 y(t)= x,\ q(t)=p
\end{cases},
\end{equation}
where $(x, p)\in\R^d\times\cI$ is the initial state, $t\in[0, T]$ the initial instant, $T>0$ the fixed finite horizon. The measurable control is (for $A\subset\mathbb{R}^m$ compact)
$$
\alpha\in{\cA}:=\left\{\alpha:[0,+\infty[\lra A\ \mbox{measurable}\right\}
$$
and the dynamics $q(\cdot)$ of the switching variable (which represents here the memory) is subject to
$$
\exists\tau\in[t, s], \ y(\tau)\in{\T_j}\Rightarrow\ q^j(s)=1;\ \ q^j(s)=p^j\ \text{otherwise}.
$$
Formally $q^j(s)=0$ means that the target ${\T_j}$ has not been visited yet in $[t, s]$ and vice versa for $q^j(s)=1$. The dynamics $f :\R^d\times A \times \I\longrightarrow \R^d$ is continuous,  bounded and Lipschitz continuous w.r.t. $x\in\R^d$ uniformly w.r.t. $(a,p)\in A\times\cI$, i.e., there exists $L>0$ such that
$$
\|f(x,a,p)-f(y,a,p)\|\le L\|x-y\|
$$
for all $(x,y)\in\R^d\times\R^d$ and $(a,p)\in A\times\cI$. By our hypotheses, for every initial state $(x,t,p)$ and control $\alpha$, the existence of a unique solution of \eqref{eq_stato} is guaranteed. Note that the number of switches of the variable $q$ is necessarily finite, hence $q$ is piecewise constant and the solution $y^\alpha_{(x,t,p)}(s)$ (or simply $y(s)$) of \eqref{eq_stato} is in the sense of absolutely continuous function. 
\par
The optimal visiting problem is then to reach, if possible, the discrete state
$
\bar p=(1,1,\dots,1)
$
(i.e. to visit all the targets) at a time $t\leq\bar t\leq T$, minimizing the cost
$$
\int_t^{\bar t}e^{-\l(s-t)}\ell(y(s), \a(s), q(s), s)ds,
$$
for a given running cost $\ell$ and a discount factor $\l>0$. 

\subsection{A hybrid-control relaxation: optimal switching}
\label{firstrel}
The optimal control problem described above requires to ``exactly touch" all the targets in an order which is not a priori given but, due to the optimization, is part of the solution itself. This makes the evolution of the discrete variable $q$ rather complicated, in particular in view of the corresponding Hamilton-Jacobi equation. We then relax the problem replacing ``exactly touch" with ``to pass as close as possible" to each target. We then assume that we can definitely get rid of some targets at any time and take into account only the remaining ones.  In doing that, we pay an additional cost depending, for instance, on the actual distance from the discarded targets. In this way, the evolution $q(\cdot)$ is no more given by system \eqref{eq_stato}, but instead, it becomes a control at our disposal. Clearly, there are some constraints: for example, for $N=4$, if $p=(1, 0 ,0, 0)$, $p'=(1, 1, 0, 0)$, $p''=(0, 1, 1, 0)$ and $p'''=(1, 1, 1, 0)$, then from $p$ we can not switch to $p''$ otherwise we lose the information about the already visited/discarded target $\cT_1$.  We can, instead, switch to $p'''$ directly.

Hence, for any $p$, we denote by $\cI_p$ the set of all possible new variables in $\cI$ after a switch from $p$:
\begin{multline*}
\cI_p=\{\tilde p\in\cI:p^i=1 \Rightarrow \tilde p^i=1\\
 \text{and}\ \exists l=1,\ldots,N:p^l=0, \; \tilde p^l=1\}.
\end{multline*}
We observe that in particular $\cI_{\bar p}=\emptyset$, where $\bar p=(1,\ldots,1)$.
\par
For a given $p$, the number of the admissible subsequent switches is at most $N-\sum_{i}p^i\leq N$. Given the state $(x, p)$ at the time $t$ with $p\neq\bar p$, the controller chooses the measurable control $\a\in\cA$, and the discrete one $q:[0,+\infty[\lra\cI$ which contains: the number $1\leq m\leq N-\sum_{i}p^i$ of switches to be performed in order to reach $\bar p$, the switching instants $t\leq t_1<t_2<\ldots<t_m\leq T$ and the switching destinations $p_1,\ldots,p_{m-1}$, $p_m=\bar p$. The destinations must satisfy 
$
p_1\in\cI_{p}, \, p_{i+1}\in\cI_{p_i},\, i=1,\ldots,m-1.
$
Then, to resume, the control at disposal is 
$$
(\a,q)=(\a, m, t_1,\ldots,t_m, p_1,\ldots,p_{m-1})=:u
$$
and note that, for any $(x, p, t)$ as above, such a string belongs to a set depending on $p$ and $t$, denoted by $\cU_{(p, t)}$.
The cost to be minimized is
\begin{multline*}
J(x, t, p, u)=\sum_{j=1}^m\Bigg(\int_{t_{j-1}}^{t_j}e^{-\l(s-t)}\ell(y(s), \a(s), p_{j-1}, s)ds\\
+e^{-\l(t_j-t)}C(y(t_j), p_{j-1}, p_j)\Bigg),
\end{multline*}
with $\l\ge0$, $p_0=p$, $t_0=t$ and $y(s)$ is the solution of \eqref{eq_stato} where $q(s)=p_{j-1}$ if $s\in[t_{j-1}, t_j]$. 
\par
We assume $\ell:\R^d\times A\times\cI\times[0, T]\lra[0, +\infty[$ bounded, continuous and uniformly continuous w.r.t. $x$ uniformly w.r.t. $a\in A$, $p\in\cI$ and $t\in[0, T]$.
Moreover $C:\R^d\times\cI\times\cI\lra[0, +\infty[$ is uniformly continuous w.r.t. $x\in\R^d$, uniformly w.r.t. $p, p'\in\cI\times\cI_p$. Note that $C(x, p, p')$ represents the switching cost from $p$ to $p'$ when the state position is $x\in\R^d$. For example, it may depend on the distance from the discarded targets, that is $C(x, p, p')=\sum_j\chi_j(p, p')d(x,\cT_j)$, where
$$
\chi_j(p, p')=\begin{cases}0,&p^j=p'^j\\
1,&\text{otherwise}
\end{cases}.
$$
The value function of the problem is
\begin{equation}
\label{eq:V-switching}
V(x, t, p)=\inf_{u\in \cU_{(p, t)}}J(x, t, p, u).
\end{equation}

\subsection{A family of optimal stopping problems}
\label{subsec:optimal-stopping}
Here, in order to better exploit the hierarchical feature of the model, we divide the optimal switching problem above into several optimal stopping subproblems, one per every switching variable $p$, suitably coupled by the stopping costs. For example, suppose $N=4$ and take $p$ such that $\sum_{i}p^i=N-1=3$ (i.e., from $p$ we can switch only to $\bar p$). Then,  for a $(x, t, p)$, the controller has only to choose $u=(\a\in\cA, \tau\in[t,T])$ and minimize the cost 
\begin{multline}
\label{primocosto}
J_p(x, t, \a, \tau)=\int_t^{\tau}e^{-\l(s-t)}\ell(y(s), \a(s), p, s)ds\\
+e^{-\l(\tau-t)}C(y(\tau), p, \bar p).
\end{multline}
Note that in this representation $p$ is fixed, that is does not change in the time interval $[t,\tau]$. Hence \eqref{primocosto} gives a time-dependent optimal stopping problem in the state space $\R^d$, whose value function is
$$
V_p(x, t)=\inf_{(\a, \tau)} J_p(x, t, \a, \tau).
$$
Now, take $p$ such that $\sum_{i}p^i=N-2=2$. Then consider the time-dependent optimal stopping problem in the state space $\R^d$ where, for a given $(x, t)$, the control is $u=(\a\in\cA, \tau\in[t,T], p'\in\cI_p)$ and the cost to be minimized is
\begin{multline}
\label{secondocosto}
J_p(x, t, \a, \tau, p')=\int_t^{\tau}e^{-\l(s-t)}\ell(y(s), \a(s), p, s)ds
\\
+e^{-\l(\tau-t)}\Big(C(y(\tau), p, p')+V_{p'}(y(\tau), \tau)\Big).
\end{multline}
Note that from $p'$ we can only switch to the final state $\bar p$, and hence $V_{p'}$ can be a priori evaluated as in the previous step. Since when $p=\bar p$, the game stops, we set $V_{\bar p}\equiv0$. Hence \eqref{primocosto} can be seen formulated as \eqref{secondocosto}. The value function is then
\begin{equation}
\label{funzionivaloritimep}
V_p(x, t)=\inf_{(\a, \tau, p')}J_p(x, t, \a, \tau, p').
\end{equation}
Proceeding backwardly in this way, we consider a suitable time-dependent optimal stopping problem in $\R^d$ for any $p\in\cI$, and we can at least formally compute the corresponding value functions $V_p$. 
\par
We will see in \S\ref{equivalenza} the equivalence between the optimal control problem formulated in \S\ref{firstrel} and the family of optimal stopping problems here formulated.

\subsection{Time-dependent optimal stopping problem: position and theoretical results}
\label{opttime}
Here we collect some theoretical results for a time-dependent optimal stopping problem with a fixed finite horizon $T>0$.  We suitably generalize to our finite horizon time-dependent model the results in \cite{BardiCapuz@incollection} for an optimal stopping problem with no time-dependence and infinite horizon feature. Hence we drop the variable $p$ in the dynamical system:
\begin{equation}
\label{system}
\begin{cases}
y'(s)=f(y(s), \a(s)),&s\in]t, T]\\
y(t)=x
\end{cases},
\end{equation}
where $x\in\R^d$, $t\in[0, T]$ and
$$
\a\in\cA:=\{\a:[0, +\infty[\lra A: \a \ \text{is measurable}\},
$$
$A\subset\R^m$ is compact, $f:\R^d\times A\lra\R^d$ is continuous, bounded and there exists $L>0$ such that
\begin{equation}
\label{hpf0}
\|f(x, a)-f(y, a)\|\leq L\|x-y\|,\quad\forall x,y\in\R^d, \ a\in A.
\end{equation}
\par
We recall the following basic estimate on the trajectory $y_{(x, t)}(\cdot; \a)$: for all $x\in\R^d$, $\a\in\cA$ and $s\in[t, T]$,
\begin{equation}
\label{firstest}
\|y_{(x, t)}(s; \a)-x\|\leq M(s-t),
\end{equation}
where $M:=\sup\{\|f(z, a)\|: (z, a)\in\R^d\times A\}$.
\par\smallskip\noindent
The cost to be minimized is
\begin{multline*}
J(x, t, \a, \tau)=\int_t^{\tau}e^{-\l(s-t)}\ell(y_{(x, t)}(s; \a), \a(s), s)ds \\
+e^{-\l(\tau-t)}\psi(y_{(x, t)}(\tau; \a), \tau),
\end{multline*}
where $\tau\leq T$ is the stopping time and $\l\geq0$ the discount factor. We assume that
\begin{itemize}
\item[-] $\psi:\R^d\times[0, T]\lra[0, +\infty[$ is bounded and uniformly continuous; 
\item[-] $\ell:\R^d\times A\times[0, T]\lra[0, +\infty[$ is bounded, continuous and such that there exists a modulus of continuity $\omega_{\ell}$  for which $|\ell(x, a, t)-\ell(y, a, t)|\leq\omega_{\ell}(\|x-y\|)$ for every $x, y\in\R^d$, $a\in A$ and $t\in[0, T]$.
\end{itemize}

The value function is
\begin{equation}
\label{funzionevaloretime}
V(x, t)=\inf_{(\a\in\cA, \tau\geq t)}J(x, t, \a, \tau).
\end{equation}
In the sequel, by $\BUC(E)$ we denote the space of bounded and uniformly continuous functions on $E\subseteq\R^n$.
\par\smallskip
\begin{proposition}
\label{contofv}
Under the previous hypotheses, $V$ as in \eqref{funzionevaloretime} is in $\BUC(\R^d\times[0, T])$.
\end{proposition}
\begin{proof} It follows from standard arguments recalling that $f$ is bounded and that for all $x, z\in\R^d$, $\a\in\cA$, $t, \tau\in[0, T]$ and $s\in[\max(t, \tau), T]$, 
\begin{multline*}
\|y_{(x, t)}(s; \a)-y_{(z,\tau)}(s; \a)\|\\
\leq e^{L(T-\max(t, \tau))}(\|x-z\|+M|t-\tau|).
\end{multline*}
\end{proof}
We have the following dynamic programming principle.
\par\smallskip
\begin{proposition}
\label{dynproggen}
Assume the hypotheses of Proposition \ref{contofv}. For every $x\in\R^d$ and $t\in[0, T]$, we have
\begin{itemize}
\item[$(i)$]$V(x, t)\leq\psi(x, t)$;
\item[$(ii)$] for every $\tilde t\geq t$, $\a\in\cA$,
\begin{multline*}
V(x, t)\leq \int_t^{\tilde t}e^{-\l(s-t)}\ell(y_{(x, t)}(s;\a), \a(s), s)ds\\
+e^{-\l(\tilde t-t)}V(y_{(x, t)}(\tilde t;\a), \tilde t);
\end{multline*}
\item[$(iii)$] for any $(x, t)$ for which the strict inequality in $(i)$ holds, there exists $t_0=t_0(x, t)>0$ such that, for every $\zeta\in[t, t+t_0]$,
$$
V(x, t)=\inf_{\a\in\cA}\Bigg(\int_t^{\zeta}e^{-\l(s-t)}\ell(y_{(x, t)}(s; \a), \a(s), s)ds
$$
$$
+e^{-\l(\zeta-t)}V(y_{(x, t)}(\zeta; \a), \zeta)\Bigg).
$$
\end{itemize}
\end{proposition}
\par
\begin{proof}
Inequality $(i)$ is clear since, in particular, 
$$
V(x, t)\leq J(x, t, \a, t)=\psi(x, t)\quad\text{for all } (x, t)\in\R^d\times[0, T].
$$
\par
For $(ii)$, fix $\a\in\cA$, $\tilde t\geq t$, $\varepsilon>0$ and let $(\tilde\a, \tilde\tau\geq\tilde t)$ be $\varepsilon$-optimum for $V(y_{(x, t)}(\tilde t; \a), \tilde t)$, that is
$$
J(y_{(x, t)}(\tilde t; \a), \tilde t, \tilde\a, \tilde\tau)\leq V(y_{(x, t)}(\tilde t; \a), \tilde t)+\varepsilon.
$$
Now define
$$
\hat\a(\tau)=\begin{cases}\a(\tau),&\tau\leq\tilde t\\
\tilde\a(\tau-\tilde t),&\tau>\tilde t
\end{cases}.
$$
Observe that, calling $z:=y_{(x, t)}(\tilde t; \a)$, we have
\begin{multline*}
V(x, t)\leq J(x, t, \hat\a, \tilde\tau)\\=\int_t^{\tilde\tau}e^{-\l(s-t)}\ell(y_{(x, t)}(s; \hat\a), \hat\a(s), s)ds\\
+e^{-\l(\tilde\tau-t)}\psi(y_{(x, t)}(\tilde\tau;\hat\a), \tilde\tau)\\
=\int_t^{\tilde t}e^{-\l(s-t)}\ell(y_{(x, t)}(s;\a), \a(s), s)ds\\+\int_{\tilde t}^{\tilde\tau}e^{-\l(s-t)}\ell(y_{(z, \tilde t)}(s;\tilde\a), \tilde\a(s), s)ds\\
+e^{-\l(\tilde\tau-t)}\psi(y_{(z, \tilde t)}(\tilde\tau; \tilde\a), \tilde\tau)\\
%=\int_t^{\tilde t}e^{-\l(s-t)}\ell(y_{(x, t)}(s;\a), \a(s), s)ds
%\\
%+e^{-\l(\tilde t-t)}\int_{\tilde t}^{\tilde\tau}e^{-\l(s-\tilde t)}\ell(y_{(z, \tilde t)}(s;\tilde\a), \tilde\a(s), s)ds
%\\
%+e^{-\l(\tilde\tau-\tilde t)}e^{-\l(\tilde t-t)}\psi(y_{(z, \tilde t)}(\tilde\tau; \tilde\a), \tilde\tau)
%\\
=\int_t^{\tilde t}e^{-\l(s-t)}\ell(y_{(x, t)}(s;\a), \a(s), s)ds
\\+e^{-\l(\tilde t-t)}\Bigg(\int_{\tilde t}^{\tilde\tau}e^{-\l(s-\tilde t)}\ell(y_{(z, \tilde t)}(s;\tilde\a), \tilde\a(s), s)ds
\\
+e^{-\l(\tilde\tau-\tilde t)}\psi(y_{(z, \tilde t)}(\tilde\tau; \tilde\a), \tilde\tau)\Bigg)\\
=\int_t^{\tilde t}e^{-\l(s-t)}\ell(y_{(x, t)}(s;\a), \a(s), s)ds\\+e^{-\l(\tilde t-t)}J(z, \tilde t, \tilde\a, \tilde\tau)\\
\leq\int_t^{\tilde t}e^{-\l(s-t)}\ell(y_{(x, t)}(s;\a), \a(s), s)ds\\+e^{-\l(\tilde t-t)}\Big(V(y_{(x, t)}(\tilde t; \a), \tilde t)+\varepsilon\Big).
\end{multline*}
Then, from the arbitrariness of $\varepsilon$, the inequality follows.
\par
Assertion $(iii)$ can be proved as in \cite{BardiCapuz@incollection}, taking into account the time variable too.
\end{proof}
For $x, \xi\in\R^d$ and $t\in[0, T]$, we define the Hamiltonian function by
$$
H(x, t, \xi)=\sup_{a\in A}\{-f(x, a)\cdot \xi -\ell(x, a, t)\}.
$$
In the sequel, by $(\cdot)_t$ and $D_x$ we denote the time derivative and the spatial gradient.
\par\smallskip
\begin{theorem}
\label{pbl1}
Under the hypotheses of Proposition \ref{dynproggen}, the value function $V$ is a viscosity solution of
\begin{equation}
\label{VIOT}
\begin{cases}
\max\{u(x, t)-\psi(x, t), \\ -u_t(x, t)+\l u(x, t)+H(x, t, D_xu(x, t))\}=0,\\ \mbox{\hspace{5.1cm}$(x, t)\in\mathbb{R}^d\times[0, T[$}\\
u(x, T)=\psi(x, T),\mbox{\hspace{3.97cm}$x\in\R^d$}
\end{cases}.
\end{equation}
\end{theorem}
\begin{proof}
%Let $(x_1, t_1)\in\R^d\times[0, T]$ be a local maximum point of $V-\phi$, $\phi\in C^1(\R^d\times[0, T])$. Then, for some $r>0$,
%$$
%V(x_1, t_1)-V(z, t)\geq\phi(x_1, t_1)-\phi(z, t)
%$$
%for every $(z, t)\in B((x_1, t_1), r)$. Fix an arbitrary $a\in A$ and let $y_{(x_1, t_1)}(\cdot)$ be the solution corresponding to the constant control $\a(\zeta)=a$ for all $\zeta$. For $\zeta$ sufficiently close to $t_1$, $(y_{(x_1, t_1)}(\zeta), \zeta)\in B((x_1, t_1), r)$ by \eqref{firstest}, and then
%$$
%\phi(x_1, t_1)-\phi(y_{(x_1, t_1)}(\zeta), \zeta)\leq V(x_1, t_1)-V(y_{(x_1, t_1)}(\zeta), \zeta).
%$$
%By $(ii)$ of Proposition \ref{dynproggen}, we obtain
%\begin{multline*}
%\phi(x_1, t_1)-\phi(y_{(x_1, t_1)}(\zeta), \zeta)\\
%\leq\int_{t_1}^{\zeta}e^{-\l(s-t_1)}\ell(y_{(x_1, t_1)}(s), \a(s), s)ds\\
%+(e^{-\l(\zeta-t_1)}-1)V(y_{(x_1, t_1)}(\zeta), \zeta).
%%\int_{t_1}^{\zeta}e^{-\l(s-t_1)}\ell(y_{(x_1, t_1)}(s), \a(s), s)ds\\
%%+e^{-\l(\zeta-t_1)}V(y_{(x_1, t_1)}(\zeta), \zeta)-V(y_{(x_1, t_1)}(\zeta), \zeta)
%\end{multline*}
%Dividing now by $\zeta-t_1$ and letting $\zeta\ra t_1$, by the differentiability of $\phi$ w.r.t. $x$ and $t$ we get
%\begin{multline*}
%-\phi_t(x_1, t_1)-D_x\phi(x_1, t_1)\cdot f(x_1, a)\\
%\leq\ell(x_1, a, t_1) -\l V(x_1, t_1).
%\end{multline*}
%Since $V(x, t)\leq\psi(x, t)$ for every $(x, t)\in\R^d\times[0, T]$ by $(i)$ of Proposition \ref{dynproggen} and $a\in A$ is arbitrary, the subsolution condition follows.
%\par
Let $(x_2, t_2)\in\R^d\times[0, T[$ be a local minimum point of $V-\phi$, that is, for some $r>0$, 
\begin{equation}
\label{max}
V(x_2, t_2)-V(z, t)\leq\phi(x_2, t_2)-\phi(z, t)
\end{equation}
for every $(z, t)\in B((x_2, t_2), r)$. If $V(x_2, t_2)=\psi(x_2, t_2)$, then, obviously,
\begin{multline*}
\max\{V(x_2, t_2)-\psi(x_2, t_2),\\
 -V_t(x_2, t_2)+\l V(x_2, t_2)+H(x_2, t_2, D_xV(x_2, t_2))\}\\
\geq V(x_2, t_2)-\psi(x_2, t_2)=0
\end{multline*}
and $V$ is a supersolution of \eqref{VIOT}. Assume then $V(x_2, t_2)<\psi(x_2, t_2)$ (the only other possibility by $(i)$ of Proposition \ref{dynproggen}). For each $\varepsilon>0$ and $\zeta\geq t_2$, by $(iii)$ of Proposition \ref{dynproggen} there exists $\bar\a\in\cA$ such that
\begin{multline}
\label{disug1v}
V(x_2, t_2)\geq\int_{t_2}^{\zeta}e^{-\l(s-t_2)}\ell(\bar y_{(x_2, t_2)}(s), \bar\a(s), s)ds\\+e^{-\l(\zeta-t_2)}V(\bar y_{(x_2, t_2)}(\zeta), \zeta)-(\zeta-t_2)\varepsilon,
\end{multline}
where $\bar y_{(x_2, t_2)}(s)=y_{(x_2, t_2)}(s; \bar\a)$ is the trajectory of \eqref{system} corresponding to $\bar\a$. Now, by the hypotheses on $\ell$ and by \eqref{firstest}, we have
\begin{equation}
\label{stimal}
|\ell(\bar y_{(x_2, t_2)}(s), \bar\a(s), s)-\ell(x_2, \bar\a(s), s)|\leq\omega_{\ell}(M(s-t_2)),
\end{equation}
and, by \eqref{hpf0} and \eqref{firstest} again,
\begin{equation}
\label{stimasuf}
\|f(\bar y_{(x_2, t_2)}(s), \bar\a(s))-f(x_2, \bar\a(s))\|\leq LM(s-t_2).
\end{equation}
By \eqref{stimal}, the integral on \eqref{disug1v} can be written as
$$
\int_{t_2}^{\zeta}e^{-\l(s-t_2)}\ell(x_2, \bar\a(s), s)ds+o(\zeta-t_2)\quad\text{as}\ \zeta\ra t_2,
$$
where $o(\zeta-t_2)$ indicates a function $g(\zeta-t_2)$ such that $\lim_{\zeta\ra t_2}g(\zeta-t_2)/(\zeta-t_2)=0$ and, in this case, $|g(\zeta-t_2)|\leq(\zeta-t_2)\omega_{\ell}(M(\zeta-t_2))$. Then, by \eqref{max} with $(z, t)=(\bar y_{(x_2, t_2)}(\zeta), \zeta)$ and by \eqref{disug1v}, we obtain
\begin{multline}
\label{int1}
\phi(x_2, t_2)-\phi(\bar y_{(x_2, t_2)}(\zeta), \zeta)-\int_{t_2}^{\zeta}e^{-\l(s-t_2)}\ell(x_2, \bar\a(s), s)ds\\
+(1-e^{-\l(\zeta-t_2)})V(\bar y_{(x_2, t_2)}(\zeta), \zeta)\geq-(\zeta-t_2)\varepsilon+o(\zeta-t_2).
\end{multline}
Moreover, by \eqref{firstest}, \eqref{stimasuf} and the fact that $\phi\in C^1$, we have
\begin{multline}
\label{int2}
\phi(x_2, t_2)-\phi(\bar y_{(x_2, t_2)}(\zeta), \zeta)=-\int_{t_2}^{\zeta}\frac{d}{ds}\phi(\bar y_{(x_2, t_2)}(s), s)ds\\
=-\int_{t_2}^{\zeta}(D_x\phi(\bar y_{(x_2, t_2)}(s), s)\cdot f(\bar y_{(x_2, t_2)}(s), \bar\a(s))\\+\phi_t(\bar y_{(x_2, t_2)}(s), s))ds\\
=-\int_{t_2}^{\zeta}(D_x\phi(x_2, s)\cdot f(x_2, \bar\a(s))+\phi_t(x_2, s))ds + o(\zeta-t_2).
\end{multline}
Putting \eqref{int2} into \eqref{int1} and adding $\pm\int_{t_2}^{\zeta}\ell(x_2, \bar\a(s), s)ds$, we get
\begin{multline}
\label{int3}
\int_{t_2}^{\zeta}\{-D_x\phi(x_2, s)\cdot f(x_2, \bar\a(s))\\
-\phi_t(x_2, s)-\ell(x_2, \bar\a(s), s)\}ds\\+\int_{t_2}^{\zeta}(1-e^{-\l(s-t_2)})\ell(x_2, \bar\a(s), s)ds\\
+(1-e^{-\l(\zeta-t_2)})V(\bar y_{(x_2, t_2)}(\zeta), \zeta)\geq-(\zeta-t_2)\varepsilon+o(\zeta-t_2).
\end{multline}
The first integral is estimated from above by
$$
\int_{t_2}^{\zeta}\sup_{a\in A}\{-D_x\phi(x_2, s)\cdot f(x_2, a)-\phi_t(x_2, s)-\ell(x_2, a, s)\}ds
$$
and the second one is $o(\zeta-t_2)$ by the hypotheses on $\ell$. Dividing \eqref{int3} by $\zeta-t_2$ and letting $\zeta\ra t_2$, we obtain
\begin{multline*}
-\phi_t(x_2, t_2)+\sup_{a\in A}\{-D_x\phi(x_2, t_2)\cdot f(x_2, a)-\ell(x_2, a, t_2)\}\\
+\l V(x_2, t_2)\geq-\varepsilon,
\end{multline*}
where we also used the continuity of $V$ and $\bar y_{(x_2, t_2)}$ at $(x_2, t_2)$ and $t_2$ respectively. Since $\varepsilon$ is arbitrary, the supersolution condition follows.
\par
The subsolution condition easily comes from the time-independent case in \cite{BardiCapuz@incollection}.
\end{proof}
\par
For the uniqueness, we show that if $u$ is a viscosity solution of \eqref{VIOT}, then
$$
u(x, t)=\inf_{\a\in\cA}J(x, t, \a, \tau^*_{(x, t)}(\a))
$$
for some $\tau^*_{(x, t)}(\a)$ such that
$$
\inf_{\a\in\cA}J(x, t, \a, \tau^*_{(x, t)}(\a))=\inf_{(\a\in\cA, \tau\geq t)}J(x, t, \a, \tau)=V(x, t),
$$
and hence $V$ is the unique viscosity solution.
\par\smallskip
\begin{lemma}
\label{proput}
Let $\Omega\subseteq\R^d$ be an open subset. For $x\in\Omega$, $t\in[0, T]$ and $\a\in\cA$, we set
$$
\tau_{(x, t)}(\a):=\min\{\inf\{\tau\geq t:y_{(x, t)}(\tau; \a)\notin\Omega\}, T\}.
$$
Then, under the hypotheses of Theorem \ref{pbl1}, for $u\in\BUC(\Omega\times[0, T])$ the following statements are equivalent:
\begin{itemize}
\item[$(i)$] $e^{-\l(s-t)}u(y_{(x, t)}(s; \a), s)-e^{-\l(\tau-t)}u(y_{(x, t)}(\tau; \a), \tau)$\\
$\displaystyle \leq\int_s^{\tau}e^{-\l(\zeta-t)}\ell(y_{(x, t)}(\zeta; \a), \a(\zeta), \zeta)d\zeta,$\\
with $ \a\in\cA, \ x\in\Omega, \ t\leq s\leq\tau<\tau_{(x, t)}(\a)$;
\item[$(ii)$]$-u_t(x, t)+\l u(x, t)+H(x, t, D_xu(x, t))\leq0$, 
\item[$(iii)$]$u_t(x, t)-\l u(x, t)-H(x, t, D_xu(x, t))\geq0$,
\end{itemize}
where $(ii)$ and $(iii)$ are understood in the viscosity sense for $(x, t)\in\Omega\times[0, T[$. 
\end{lemma}
\begin{proof}
It is a careful adaption to the time-dependent case of the one in \cite{BardiCapuz@incollection}.
\end{proof}
\par\smallskip
\begin{theorem}
\label{unicmodo1}
Let $u\in\BUC(\R^d\times[0, T])$ be a viscosity solution of \eqref{VIOT}. Then, under the hypotheses of Theorem \ref{pbl1},
$$
u(x, t)=\inf_{\a\in\cA}J(x, t, \a, \tau^*_{(x, t)}(\a))=V(x, t)
$$
for every $(x, t)\in\R^d\times[0, T]$, where
\begin{multline}
\label{star}
\tau^*:=\tau^*_{(x, t)}(\a)=\inf\{\tau\in[t, T]:\\u(y_{(x, t)}(\tau; \a), \tau)=\psi(y_{(x, t)}(\tau; \a), \tau)\}.
\end{multline}
\end{theorem}
\begin{proof}
At first we observe that, since $u(x, T)=\psi(x, T)$ for every $x\in\R^d$, the set in \eqref{star} is always non empty, and hence $\tau^*\leq T$ always exists. Now let $u\in\BUC(\R^d\times[0, T])$ be a viscosity solution of \eqref{VIOT} and consider the open set $\cC=\{(x, t)\in\R^d\times[0, T[:u(x, t)<\psi(x, t)\}$. Similarly to \cite{BardiCapuz@incollection}, it can be proved that
\begin{equation}
\label{tr1}
u(x, t)\leq\psi(x, t),\quad(x,t)\in\R^d\times[0, T]
\end{equation}
and that
\begin{multline}
\label{tr2}
-u_t(x, t)+\l u(x, t) + H(x, t, D_xu(x, t))\leq0,\\ (x, t)\in\R^d\times[0, T[,
\end{multline}
\begin{equation}
\label{tr3}
-u_t(x, t)+\l u(x, t) + H(x, t, D_xu(x, t))=0,\ \ \ (x, t)\in\cC,
\end{equation}
in the viscosity sense (the validity of \eqref{tr1} at $t=T$ comes from the boundary condition in \eqref{VIOT}).
%By contradiction, suppose that $u(x_0, t_0)>\psi(x_0, t_0)$ at some $(x_0, t_0)\in\R^d\times[0, T]$. Then, by continuity,
%\begin{equation}
%\label{tr4}
%u(x, t)>\psi(x, t)\quad\text{for every}\ (x, t)\in B((x_0, t_0), \d),\ \d>0.
%\end{equation}
%It can be easily proved that $u-\phi$ has a local maximum at some point $(\bar x, \bar t)\in B((x_0, t_0), \d)$ for some $\phi\in C^1(\R^d\times[0, T])$, so that, since $u$ is a visc. solution of \eqref{VIOT},
%\begin{multline*}
%\max\{u(\bar x, \bar t)-\psi(\bar x, \bar t),\\
%-u_t(\bar x, \bar t)+\l u(\bar x, \bar t)+H(\bar x, \bar t, D_x\phi(\bar x, \bar t))\}\leq0.
%\end{multline*}
%This contradicts \eqref{tr4} and hence \eqref{tr1} holds.
%\par
%By \eqref{tr1}, the inequality \eqref{tr2} immediately follows since $u$ is a viscosity solution of \eqref{VIOT}.
%\par
%To prove \eqref{tr3}, it is sufficient to show that
%\begin{equation}
%\label{tr5}
%-u_t(x, t)+\l u(x, t)+H(x, t, D_xu(x, t))\geq0,\quad(x, t)\in\cC,
%\end{equation}
%in the viscosity sense. At any local minimum $(x_1, t_1)\in\cC$ of $u-\phi$, $\phi\in C^1(\R^d\times[0, T])$, by \eqref{VIOT} we have
%\begin{multline*}
%\max\{u(x_1, t_1)-\psi(x_1, t_1),\\
%-u_t(x_1, t_1)+\l u(x_1, t_1)+H(x_1, t_1, D_x\phi(x_1, t_1))\}\geq0,
%\end{multline*}
%and \eqref{tr5} is proved since $(x_1, t_1)\in\cC$. 
%\par
Now we apply Lemma \ref{proput} with $\Omega=\R^d$, $s=t$ and, by \eqref{tr1}, \eqref{tr2}, we get
\begin{multline*}
u(x, t)\leq e^{-\l(\tau-t)}u(y_{(x, t)}(\tau; \a), \tau)\\
+\int_t^{\tau}e^{-\l(\zeta-t)}\ell(y_{(x, t)}(\zeta; \a), \a(\zeta), \zeta)d\zeta\\
\leq e^{-\l(\tau-t)}\psi(y_{(x, t)}(\tau; \a), \tau)\\
+\int_t^{\tau}e^{-\l(\zeta-t)}\ell(y_{(x, t)}(\zeta; \a), \a(\zeta), \zeta)d\zeta
\end{multline*}
for all $t\leq\tau\leq T$ and $\a\in\cA$. Then
$$
u(x, t)\leq\inf_{(\a\in\cA, \tau\geq t)}J(x, t, \a, \tau)=V(x, t).
$$
For the reverse inequality, assume at first $(x, t)\notin\cC$. In this case, $u(x, t)=\psi(x, t)$ and $\tau^*=t$. Then
\begin{multline*}
u(x, t)=\psi(x, t)=J(x, t, \a, \tau^*)\\
\geq\inf_{(\a\in\cA, \tau\geq t)}J(x, t, \a, \tau)=V(x, t).
\end{multline*}
Now suppose $(x, t)\in\cC$, so that \eqref{tr3} holds. Applying Lemma \ref{proput} with $\Omega=\cC$ and $s=t$, we obtain
\begin{multline*}
u(x, t)=\inf_{\a\in\cA}\Bigg(e^{-\l(\tau-t)}u(y_{(x, t)}(\tau; \a), \tau)\\
+\int_t^{\tau}e^{-\l(\zeta-t)}\ell(y_{(x, t)}(\zeta;\a), \a(\zeta), \zeta)d\zeta\Bigg)
\end{multline*}
for every $t\leq\tau<\tau^*\leq T$. Letting $\tau\rightarrow\tau^*$, we get
\begin{multline*}
u(x, t)=\inf_{\a\in\cA}J(x, t, \a, \tau^*)\geq\inf_{(\a\in\cA, \tau\in[t, \tau^*])}J(x, t, \a, \tau)\\ \geq\inf_{\a\in\cA, \tau\in[t, T])}J(x, t, \a, \tau)=V(x, t)
\end{multline*}
since $u(y_{(x, t)}(\tau^*; \a), \tau^*)=\psi(y_{(x, t)}(\tau^*; \a), \tau^*)$.
\end{proof}
%\begin{remark}
%For some $(\tilde x, \tilde t)$, we may have $(y_{(\tilde x, \tilde t)}(\tilde\tau; \a), \tilde\tau)\in\cC$ for all $\tilde t<\tilde\tau<T$ and $\a\in\cA$. This implies that $u(y_{(\tilde x, \tilde t)}(\tilde\tau; \a), \tilde\tau)<\psi(y_{(\tilde x, \tilde t)}(\tilde\tau; \a), \tilde\tau)$ for all $\tilde t<\tilde\tau<T$ and $\a\in\cA$. From the previous arguments, it follows that
%\begin{multline*}
%u(\tilde x, \tilde t)<e^{-\l(\tau-\tilde t)}\psi(y_{(\tilde x, \tilde t)}(\tau; \a), \tau)\\
%+\int_{\tilde t}^{\tau}e^{-\l(\zeta-\tilde t)}\ell(y_{(\tilde x, \tilde t)}(\zeta; \a), \a(\zeta), \zeta)d\zeta
%\end{multline*}
%for all $\tau\geq\tilde t$ and $\a\in\cA$. Then, for such $(\tilde x, \tilde t)$, no finite optimal stopping time exists until $T$.
%\end{remark}
\subsection{Equivalence of the two models}
\label{equivalenza}
Next step is to show the equivalence between the optimal switching problem and the family of the optimal stopping ones, i.e.,  $V(x, t, p)=V_p(x, t)$ for every $(x, t, p)\in\R^d\times[0, T]\times\cI$. Here, and in the sequel, $V$ is the value function defined in \eqref{eq:V-switching} and $V_p$ is the value function defined backwardly as in \eqref{funzionivaloritimep}.
\par\smallskip
\begin{proposition}
\label{cont1}
Under the hypotheses of \S\ref{firstrel}, we have
\begin{itemize}
\item[$(i)$] $V\in\BUC(\mathbb{R}^d\times[0, T])$ for every $p\in\cI$;
\item[$(ii)$] for every $p$, the value functions $V_p$ are bounded and uniformly continuous too. 
\end{itemize}
\end{proposition} 
\begin{proof}
For $(i)$, we may act as in Proposition \ref{contofv} in \S\ref{opttime}.
For $(ii)$, by the backward definition of $V_p$, as in \S\ref{subsec:optimal-stopping}, note that at the levels $p$ with $\sum_{i}p^i=N-1$, the stopping cost is just $C$ and hence does not depend on the value function $V_{p'}$ at lower levels $p'\in\cI_p$. For higher levels $p$ such that $\sum_{i}p^i<N-1$, let us define
\begin{equation}
\label{costopsi}
\psi_p(x, t):=\inf_{p'\in\cI_p}(C(x, p, p')+V_{p'}(x, t)),\quad x\in\R^d,
\end{equation}
and recalling that $\ell$ does not depend on $p'\in\cI_p$, we have
\begin{multline*}
V_p(x,t)=\inf_{(\a, \tau)}\Bigg(\int_t^{\tau}e^{-\l(s-t)}\ell(y(s), \a(s), p, s)ds\\
+e^{-\l(\tau-t)}\psi_p(y(\tau), \tau)\Bigg).
\end{multline*}
By the backward definition of $V_p$ at every level $p\in\cI$, the stopping cost $\psi_p(x, t)$ can be assumed as known and hence, in particular, bounded and uniformly continuous. Again, the thesis comes from the results in \S\ref{opttime}.
\end{proof}
\par\smallskip
\begin{proposition}
\label{dynprog}
Under the hypotheses of Proposition \ref{cont1}, for all $x\in\mathbb{R}^d$, $t\in[0, T]$ and $p\in\cI$, we have
\begin{multline*}
V(x, t, p)=\inf_{(\a, \tau, p'\in\cI_p)}\Bigg(\int_t^{\tau}e^{-\l(s-t)}\ell(y(s), \a(s), p, s)ds\\
+e^{-\l(\tau-t)}\Big(C(y(\tau), p, p')+V_{p'}(y(\tau), \tau)\Big)\Bigg).
\end{multline*}
\noindent
As a consequence, $V(x,t,p)=V_p(x,t)$ for all $(x,t,p)$.
\end{proposition}
\begin{proof}
We follow a procedure as the one used in \S\ref{subsec:optimal-stopping}. Consider $p$ with $\sum_{i}p^i=N-1$. By definition we have
\begin{multline*}
V(x, t, p)=\inf_{(\a, \tau)}\Bigg(\int_t^{\tau}e^{-\l(s-t)}\ell(y(s), \a(s), p, s)ds\\
+e^{-\l(\tau-t)}C(y(\tau), p, \bar p)\Bigg)=V_p(x, t)
\end{multline*}
for every $(x, t)\in\mathbb{R}^d\times[0, T]$ since $V(\cdot, \cdot, \bar p)=V_{\bar p}(\cdot, \cdot)\equiv0$. Consider now $p$ with $\sum_{i}p^i=N-2$. 
We need to show that
\begin{multline}
\label{inf1}
V(x, t, p)=\inf_{(\a, \tau, p'\in\cI_p)}\Bigg(\int_t^{\tau}e^{-\l(s-t)}\ell(y(s), \a(s), p, s)ds\\
+e^{-\l(\tau-t)}\Big(C(y(\tau), p, p')+V_{p'}(y(\tau), \tau)\Big)\Bigg).
\end{multline}
We recall that, calling $x_1:=y^\a_{(x, t, p)}(t_1)$, we have
\begin{multline}
\label{inf2}
V(x, t, p)=\\
\inf_{(\a, t\leq t_1\leq t_2, p_1\in\cI_p)}\Bigg{(}\int_t^{t_1}e^{-\l(s-t)}\ell(y^\a_{(x, t, p)}(s), \a(s), p, s)ds \\
+e^{-\l(t_1-t)}C(x_1, p, p_1)\\
+\int_{t_1}^{t_2}e^{-\l(s-t)}\ell(y^\a_{(x, t, p_1)}(s), \a(s), p_1, s)ds
\\
+e^{-\l(t_2-t)}C(y^\a_{(x, t, p_1)}(t_2), p_1, \bar p)\Bigg{)}.
\end{multline}
So we have to prove that the $\inf$ in \eqref{inf1} coincides with the $\inf$ in \eqref{inf2}. At first we show the inequality $(\leq)$. For every $x\in\mathbb{R}^d$, $t\in[0, T]$, $\a\in\cA$ and $p_1\in\cI_p$, we have
\begin{multline*}
\int_t^{t_1}e^{-\l(s-t)}\ell(y^\a_{(x, t, p)}(s), \a(s), p, s)ds\\
+e^{-\l(t_1-t)}C(x_1, p, p_1)
\\
+\int_{t_1}^{t_2}e^{-\l(s-t)}\ell(y^\a_{(x, t, p_1)}(s), \a(s), p_1, s)ds\\
+ e^{-\l(t_2-t)}C(y^\a_{(x, t, p_1)}(t_2), p_1, \bar p)\\
%=\int_t^{t_1}e^{-\l(s-t)}\ell(y^\a_{(x, t, q)}(s), \a(s), p, s)ds
%\\
%+e^{-\l(t_1-t)}C(x_1, p, p_1)\\
%+e^{-\l(t_1-t)}\int_{t_1}^{t_2}e^{-\l(s-t_1)}\ell(y^\a_{(x_1, t_1, p_1)}(s), \a(s), p_1, s)ds\\
%+e^{-\l(t_1-t)}e^{-\l(t_2-t_1)}C(y^\a_{(x_1, t_1, p_1)}(t_2),p_1, \bar p)\\
%=\int_t^{t_1}e^{-\l(s-t)}\ell(y^\a_{(x, t, p)}(s), \a(s), p, s)ds\\+e^{-\l(t_1-t)}C(x_1, p, p_1)
%+e^{-\l(t_1-t)}\Bigg(\int_{t_1}^{t_2}e^{-\l(s-t_1)}\ell(y^\a_{(x_1, t_1, p_1)}(s), \a(s), p_1, s)ds\\
%+e^{-\l(t_2-t_1)}C(y^\a_{(x_1, t_1, p_1)}(t_2), p_1, \bar p)\Bigg)\\
=\int_t^{t_1}e^{-\l(s-t)}\ell(y^\a_{(x, t, p)}(s), \a(s), p, s)ds\\
+e^{-\l(t_1-t)}C(x_1, p, p_1)+e^{-\l(t_1-t)}J(x_1, t_1, p_1, \a, t_2, \bar p)\\
\geq\int_t^{t_1}e^{-\l(s-t)}\ell(y^\a_{(x, t, p)}(s), \a(s), p, s)ds\\
+e^{-\l(t_1-t)}\Big(C(x_1, p, p_1)+V(x_1, t_1, p_1)\Big)\\
=\int_t^{t_1}e^{-\l(s-t)}\ell(y^\a_{(x, t, p)}(s), \a(s), p, s)ds\\
+e^{-\l(t_1-t)}\Big(C(x_1, p, p_1)+V_{p_1}(x_1, t_1)\Big).
\end{multline*}
Take the inf on $(\a, t_1, p_1\in\cI_p)$ to get $V(x, t, p)\geq V_p(x, t)$.
\par
The reverse inequality can be proved similarly and hence we can conclude proceeding backwardly.
\end{proof}

\subsection{Optimality condition for $V_p$ in PDE form}
For $x, \xi\in\R^d$, $t\in[0, T]$ and $p\in\cI$, we define the Hamiltonian function by
$$
H^p(x, t, \xi)=\sup_{a\in A}\{-f(x, a, p)\cdot \xi -\ell(x, a, p, t)\}.
$$
\par\smallskip
\begin{theorem}
In the hypotheses of Proposition \ref{dynprog}, for any $p\in\cI$ the value function $V_p$ is the unique bounded and uniformly continuous viscosity solution $u$ of ($\psi_p$ as in \eqref{costopsi})
\begin{equation}
\label{V}
\begin{cases}
\max\{u(x, t)-\psi_p(x, t),\\
-u_t(x, t)+\l u(x, t)+H^p(x, t, D_xu(x, t))\}=0,\\ \mbox{\hspace{5.1cm}$(x, t)\in\mathbb{R}^d\times[0, T[$}\\
u(x, T)=\psi_p(x, T), \mbox{\hspace{3.8cm}$x\in\R^d$}
\end{cases}.
\end{equation}
\end{theorem}
\begin{proof}
See Theorem \ref{pbl1} and \ref{unicmodo1} in \S\ref{opttime}.
\end{proof}
\par\smallskip
\begin{theorem}
The family of functions $V:=\{V_p:p\in\cI\}$ is the unique family of bounded and uniformly continuous functions $U:=\{u_p:p\in\cI\}$ that solves the problem
\begin{equation}
\label{P}
\left\{
\begin{array}{ll}
\displaystyle
\text{for any $p\in\cI$, $u_p$ is the unique visc. sol. of \eqref{V} with}\\
\displaystyle
\psi_p\ \mbox{repl. by } \psi_p^U(x, t):=\inf_{p'\in\cI_p}(C(x, p, p')+u_{p'}(x, t)),\\
\displaystyle
u_{\bar p}=0
\end{array}.
\right.
\end{equation}
\end{theorem}
\begin{proof}
Note that $\psi^V_p=\psi_p$ as in (\ref{costopsi}). For $p\in\cI$ such that $\sum_ip^i=N-1$, we have that the problem \eqref{P} is the same as \eqref{V} because $\psi_p=\psi_p^U$. Then, by Theorem \ref{unicmodo1} in \S\ref{opttime}, $V_p=u_p$. Then, if $p\in\cI$ is such that $\sum_ip^i=N-2$, we also have $\psi_p=\psi_p^U$ and again $u_p=V_p$. We then conclude backwardly.
\end{proof}

\section{Conclusion}
\label{conclusion}
We discuss the well-position of the time-dependent hybrid control problem to model an optimal visiting problem. To such a goal, we also adapt some techniques, already used for the non time-dependent optimal stopping problem, to the time-dependent one, which seems to be not discussed elsewhere in all its details, and whose formulation is not always straightforward. As additional point of relevance, these results are preliminary for the building of the mean-field case. The same ideas can be extended to other standard cases as the presence of automatic switching or the stationary case.  Getting rid of the relaxation in \S\ref{firstrel}, is also worth investigating.

\end{document}